\documentclass[12pt, a4paper]{article}
\usepackage{pdfsync}

\usepackage{amsmath, amsthm, amssymb}\multlinegap=0pt
\usepackage[latin1]{inputenc}
\usepackage[francais,english]{babel}
\numberwithin{equation}{section}
\theoremstyle{plain}
\newtheorem{thm}{Theorem}[section]

\newtheorem{prop}[thm]{Proposition}
\newtheorem{lem}[thm]{Lemma}

\theoremstyle{definition}

\theoremstyle{remark}
\newtheorem{rem}[thm]{Remark}

\newcommand\Leb{{\cal L}}

\newcommand{\R}{\mathbb{R}}
\newcommand\N{{\mathbb N}}

\let\eps\varepsilon

\newcommand\pref[1]{(\ref{#1})}

\DeclareMathOperator{\spt}{spt}

\DeclareMathOperator{\argmax}{argmax}

\def\<#1,#2>{\left<#1,#2\right>}

\title{On some systems controlled by the structure of their memory}
\author {G. Buttazzo  \thanks{\scriptsize Dipartimento di Matematica, Universit\`a di Pisa, Largo B. Pontecorvo, 5, 56127 Pisa, ITALY
{\texttt {buttazzo@dm.unipi.it}}}, G. Carlier, R. Tahraoui \thanks{\scriptsize CEREMADE, UMR CNRS 7534, Universit\'e Paris IX Dauphine, Pl. de Lattre de Tassigny, 75775 Paris Cedex 16, FRANCE
\texttt{carlier@ceremade.dauphine.fr}, \texttt{tahraoui@ceremade.dauphine.fr}}}

\begin{document}
\maketitle

\begin{abstract}
We consider an optimal control problem governed by an ODE with memory playing  the role of a control. We show the existence of an optimal solution and derive some necessary optimality conditions. Some examples  are then discussed.
\end{abstract}



\section{Introduction}\label{intro}

The present paper deals with the optimal control of equations of the form
\begin{equation}\label{state0}
\dot{x}(t)=\<f(.,x(.)),\nu_t>=\int_0^t f(s,x(s))\,d\nu_t(s),\ t\in[0,1],\ x(0)=x_0,
\end{equation}
where the control is a family of probability measures $t\mapsto \nu_t$ such that $\nu_t$ is supported by $[0,t]$ (nonanticipativity). 

The special case where $\nu_t=\delta_{\theta(t)}$ is a Dirac mass corresponds to  the deviated equation
\begin{equation}\label{statedelayed}
\dot{x}(t)=f\big(\theta(t),x(\theta(t))\big), \; x(0)=x_0,
\end{equation}
where the control is the deviated function $\theta$ satisfying $\theta(t)\leq t$ for all $t\in[0,1]$. We shall see that the optimal control of \pref{state0} is the natural relaxation of that of the deviated equation \pref{statedelayed}. 

We will consider the minimization of the functional
\begin{equation}\label{ocnu}
\int_0^1 j(t,x(t))\,dt + h(x(1)) + \int_0^1\Big(\int_0^1 g(t,s)\,d\nu_t(s)\Big)\,dt
\end{equation}
where $x$ is related to $\nu$ by the state equation \pref{state0}. An example of relevant function $g$ is $g(t,s)=\lambda|t-s|^p$ with $p\ge1$ and $\lambda\ge0$. In this case the last term in the previous functional is $\lambda\int_0^1 W_p^p(\delta_t,\nu_t)\,dt$ where $W_p(\delta_t, \nu_t)$ is the $p$-Wasserstein distance between $\nu_t$ and the Dirac measure at $t$
\[W_p^p(\delta_t,\nu_t)=\int_{[0,t]}|t-s|^p\,d\nu_t(s).\]
The interpretation of the Wasserstein term in the functional is therefore a penalization of long-term memory.

\smallskip

Dynamics with lags  as \pref{statedelayed} or with more general memory structure as \pref{state0} arise in many different settings in engineering, economics,  biology, modelling of financial time series... It is typically the case when studying the optimal performances of  a system in which the response to a given input occurs not instantaneously but only after a certain elapse of time. We refer for instance to the classical book of Bellman and Cooke \cite{bc} for a general overview of such functional equations. A natural motivation for the optimal control of the delayed equation \pref{statedelayed} where the deviation $\theta$ is the control is then as follows. Imagine some natural or industrial process is known to obey the delayed equation \pref{statedelayed} but with an \emph{unknown} delay function $\theta$ (think of the incubation period of some disease). If now some-possibly noisy-observation $x_0$ is available then the simplest way to \emph{estimate} $\theta$ is by \emph{least} squares, that is by minimizing
\[\int_0^1 \vert x(t)-x_0(t)\vert^2 dt\]
with respect to $\theta$, $x$ being linked to $\theta$ by \pref{statedelayed}.  As already mentioned, this problem needs to be relaxed in a suitable way, which, roughly speaking, amounts to replace \pref{statedelayed} by \pref{state0} (see section \ref{exist} for a precise statement). Another motivation for related problems in economics can be found for instance in \cite{jkt}. 

\smallskip

Now we claim that considering directly (i.e. without involving relaxation) the optimal control of \pref{state0} by $\nu_t$ is natural as well. Indeed, it is customary in time-series analysis to consider \emph{autoregressive} dynamics. In discrete time, an autoregressive process is a stochastic process that satisfies a relation of the form:
\[X_t=\sum_{k=0}^K \alpha_k f(t-k, X_{t-k})+\eps_t\]
where $\eps_t$ are independent and identically distributed random variables. Such processes thus have some memory (the range of the memory being the integer $K$, called the order of the process), and the fact that the previous dynamics is given by a convolution captures some stationarity of the memory structure. This is a particular case of
\[X_t=\sum_{s\leq k} \nu_{t}(s) f(s, X_{s})+\eps_t.\]
 Formally passing to continuous time in the previous equation strongly suggests that estimating the memory structure  of the process by least squares naturally leads to solve a problem of the form \pref{state0}-\pref{ocnu}.

\smallskip

Let us insist here on the fact that the optimal control problem \pref{state0}-\pref{ocnu} in which the memory structure is unknown and acts as a control is, as far as we know, somehow unusual. For variational or optimal control problems where a deviation or a memory structure is given and one looks for a classical optimal control, we refer to \cite{st}, \cite{cht}, \cite{ct} (necessary optimality conditions) or \cite{cthjb} (dynamic programming approach leading to an Hamilton-Jacobi equation in infinite dimensions).

\smallskip

The paper is organized as follows. Section \ref{prel} is devoted to some preliminary results. In section \ref{exist}, we prove existence of solutions for \pref{ocnu} and show that \pref{ocnu} is the natural relaxation of an optimization problem posed over deviation functions. In section \ref{pmp}, we establish optimality conditions. Finally section \ref{exampl} is devoted to some examples.

\section{Preliminaries}\label{prel}

Let $(\nu_t)_{t\in[0,1]}$ be a Borel family of probability measures (meaning that $t\mapsto\int g\,d\nu_t$ is Borel for every continuous $g$) such that $\nu_t([0,t])=1$ for every $t$. Let  $f \in C^0([0,1]\times \R^d, \R^d)$ satisfy the Lipschitz condition that there is a $k\geq 0$ such that:
\begin{equation}\label{flip}
|f(t,x)-f(t,y)|\le k|x-y|\qquad\forall t\in[0,1],\ \forall(x,y)\in\R^d\times\R^d.
\end{equation}

The next result gives the existence and uniqueness of a solution to the state equation \pref{state0}:

\begin{prop}\label{cauchylip}
Let $f\in C^0([0,1]\times\R^d,\R^d)$ satisfy the Lipschitz assumption \pref{flip}, $(\nu_t)_{t\in[0,1]}$ be as above and $x_0\in\R^d$. Then the Cauchy problem
\begin{equation}\label{cauchy}
\dot{x}(t)=\<f(., x(.)), \nu_t>,\; \forall t\in[0,1],\; x(0)=x_0
\end{equation}
admits a unique continuous solution $x$. Moreover $x$ is Lipschitz continuous and satisfies $\|x\|_{W^{1,\infty}}\le M$ for a constant $M$ that only depends on $k$, $|x_0|$ and $\sup_{t\in[0,1]}|f(t,0)|$.
\end{prop}

\begin{proof}
We proceed as in \cite{ct}. Let $\lambda>k$ and equip $C^{0}([0,1], \R^d)$ with the norm
\[\|x\|_{\lambda}:=\sup_{t\in [0,1]} e^{-\lambda t}|x(t)|.\]
For every $x\in C^{0}([0,1], \R^d)$, define the continuous function $Tx$ by
\[Tx(t):=x_0+\int_0^t\<f(.,x(.)),\nu_s>\,ds\qquad\forall t\in[0,1].\]
For $x$ and $y$ continuous and $t\in [0,1]$, one deduces from \pref{flip} and the nonanticipativity of $(\nu_t)_t$:
\[|Tx(t)-Ty(t)|\le k\|x-y\|_{\lambda}\int_0^t\int_0^s e^{\lambda\tau}\,d\nu_s(\tau)\,ds\le\frac{k}{\lambda}\|x-y\|_{\lambda}e^{\lambda t}.\]
One then deduces that $T$ is a contraction and therefore has a unique fixed point in $C^0$. The second claim easily follows. 
\end{proof}

In what follows,  $\Leb^1$ denotes the Lebesgue measure on $[0,1]$. Let $(\nu_t)_t$ be admissible for our problem which means that it is a Borel family of probability measures such  that $\nu_t([0,t])=1$ for every $t$. It will be convenient in the sequel to define $\gamma:=\nu_t \otimes \Leb^1$ that is the probability measure on $[0,1]^2$ defined by
\begin{equation}\label{defigamma}
\int_{[0,1]^2}\varphi(t,s)\,d\gamma(t,s)=\int_0^1\Big(\int_0^1\varphi(t,s)\,d\nu_t(s)\Big)\,dt
\end{equation}
for every $\varphi\in C^0([0,1]^2,\R)$. The admissibility of $(\nu_t)_t$ is equivalent to require  that $\gamma:=\nu_t \otimes \Leb^1$ belongs to the set 
\begin{equation}\label{defdeGamma}
\Gamma:=\{ \gamma \mbox{ probability on } [0,1]^2 \; : \; \gamma(T)=1, \; {\pi_1}_\# \gamma=\Leb^1\}.
\end{equation}
where $T$ is the triangle $T:=\{(t,s)\in [0,1]^2 \; : \; s\leq t\}$ and ${\pi_1}_\# \gamma$ denotes  the first  marginal of $\gamma$. Given $\gamma\in \Gamma$, the disintegration Theorem (see \cite{dm}) enables one to disintegrate $\gamma$ as $\gamma=\nu_t \otimes \Leb^1$ with $(\nu_t)_t$ admissible. In the sequel, under the assumptions of  proposition \ref{cauchylip}, the solution of the Cauchy problem \pref{cauchy} will be denoted $x_\gamma$. Let us also remark that $x_\gamma$ can be characterized by the weak form of \pref{cauchy} that can be conveniently written in terms of $\gamma$ as:
\begin{equation}\label{weakform}
\int_0^1\dot{\varphi}(t)\cdot x_\gamma(t)\,dt
=-\int_{[0,1]^2}\varphi(t)\cdot f(s,x_{\gamma}(s))\,d\gamma(t,s),\quad x_{\gamma}(0)=x_0
\end{equation}
for every $\varphi\in C_c^1((0,1),\R^d)$. The other obvious advantage of formulating the problem in terms of $\gamma\in\Gamma$ is that $\Gamma$ is weakly* compact.

In the sequel, we shall always assume that $f\in C^0([0,1]\times\R^d,\R^d)$ satisfy the Lipschitz assumption \pref{flip} so that $x_\gamma$ is well-defined for every $\gamma\in\Gamma$. 

\begin{lem}\label{cvforte}
If $(\gamma_n)\in\Gamma^\N$ weakly* converges to $\gamma$, then $x_{\gamma_n}$ converges to $x_\gamma$ in $C^0([0,1],\R^d)$. 
\end{lem}

\begin{proof}
Setting $x_n:=x_{\gamma_n}$, it follows from proposition \ref{cauchylip} and Ascoli-Arzel\`a's theorem that $(x_n)$ is precompact in $C^0$. Let $y$ be the uniform limit of some (not relabeled) subsequence. Let $\varphi\in C_c^1((0,1),\R^d)$, one has for every $n$
\begin{equation}\label{weakformn}
\int_0^1\dot{\varphi}(t)\cdot x_n(t)\,dt
=-\int_{[0,1]^2}\varphi(t)\cdot f(s,x_n(s))\,d\gamma_n(t,s),\ x_n(0)=x_0.
\end{equation}
Since $x_n$ converges uniformly to $y$, and $\gamma_n$ converges weakly* to $\gamma$, passing to the limit in \pref{weakformn} one deduces from the continuity of $f$ that $y=x_{\gamma}$, and by a standard compactness argument we deduce that the whole sequence converges to $x_\gamma$.
\end{proof}

\section{Existence of optimal controls and relaxation}\label{exist}

We are now interested in the optimization problem 
\begin{equation}\label{optigamma}
\inf_{\gamma\in \Gamma}  J(\gamma)
\end{equation}
where
\[J(\gamma):=\int_0^1 j(t, x_{\gamma}(t))\,dt + h(x_\gamma(1))+ \int_{T} g(t,s)\;d\gamma(t,s).\]

\begin{thm}\label{exthm}
We make the following assumptions:
\begin{itemize}
\item the function $j:[0,1]\times\R^d\to[0,+\infty]$ is Borel measurable, nonnegative and $j(t,\cdot)$ is lower semicontinuous for a.e. $t\in[0,1]$;
\item the function $h:\R^d\to[0,+\infty]$ is nonnegative and lower semicontinuous;
\item the function $g: T \to [0, +\infty]$ is nonnegative and lower semicontinuous on $T$;
\item there exists $\gamma_0\in \Gamma$ such that $J(\gamma_0)<+\infty$.
\end{itemize}
Then the optimal control problem \eqref{optigamma} has a finite value and admits a solution.
\end{thm}

\begin{proof}
Thanks to lemma \ref{cvforte}, one easily deduces from the assumptions above that $J$ is weakly* lower-semicontinuous (for the lower semicontinuity of the last term express $g$ as the supremum of continuous functions) and the existence claim then follows from the weak* compactness of $\Gamma$.
\end{proof}

Let us define the set of admissible deviations :
\[\Theta:=\{\theta:[0,1]\to[0,1]\mbox{ Borel},\ \theta(t)\le t,\mbox{ for a.e. }t\in[0,1]\}\]
and the optimal deviation problem
\begin{equation}\label{optidev}
\inf_{\theta\in\Theta}F(\theta),
\quad\mbox{with }F(\theta):=J(\delta_{\theta}\otimes\Leb^1)
\end{equation}
that is the restriction of \pref{optigamma} to deviation (or delay) functions.   Then \pref{optigamma} is the relaxation of \pref{optidev} in the following sense.

\begin{thm}\label{relax}
If $j\in C^0([0,1]\times \R^d)$, $h\in C^0(\R^d)$ and $g\in C^0(T)$, then for every $\gamma\in \Gamma$, there exists a sequence $\theta_n$ in $\Theta$ such that $F(\theta_n)$ converges to $J(\gamma)$. In particular
\[\min_{\gamma\in\Gamma} J(\gamma)=\inf_{\theta\in\Theta}F(\theta).\]
\end{thm}

\begin{proof}
It follows from lemma \ref{cvforte} that $J$ is weakly* continuous. Fix now $\gamma\in\Gamma$; it is a well-known result in the theory of Young measures (see for instance Theorem 9.3 in \cite{amb}) that there exists a sequence of Borel maps $\sigma_n:[0,1]\mapsto[0,1]$ such that $\delta_{\sigma_n}\otimes \Leb^1$ converges weakly* to $\gamma$. Defining $\theta_n(t):=\min\{\sigma_n(t),t\}$ so that $\theta_n\in\Theta$, it is obviously enough to prove that $\delta_{\theta_n}\otimes\Leb^1$ converges weakly* to $\gamma$ to conclude. Let $\varphi\in C^0([0,1]^2)$ and let us denote by $\omega_\varphi$ the modulus of continuity of $\varphi$. For every  $\delta >0$, we have
\[\Big|\int_0^1\varphi(t,\theta_n(t))-\varphi(t,\sigma_n(t))\,dt\Big|
\le2\|\varphi\|_{\infty}\Leb^1(\{t\ :\ \sigma_n(t)\ge t+\delta\})+\omega_{\varphi}(\delta).\]
From the weak* convergence of $\delta_{\sigma_n}\otimes\Leb^1$ to $\gamma$ and from the fact that the support of $\gamma$ is included in the triangle $T$, it is easy to deduce that for every $\delta>0$, $\Leb^1(\{t\ :\ \sigma_n(t)\ge t+ \delta\})$ tends to $0$. The weak*convergence of $\delta_{\theta_n}\otimes\Leb^1$ to $\gamma$ then follows directly.
\end{proof}

It is easy to build from the previous result  examples where there is no optimal delay function  for \pref{optidev}. Indeed, let us take $d=1$, $f(x)=x$, $x_0=1$, $h=0$, $j(t,x)=|x-x^*(t)|^2$ with $x^*(t)=-1+2e^{t/2}$ and $g(t,s)=s(t-s)$. By construction $x^*=x_{\gamma^*}$ for $\gamma^*=(\delta_0+\delta_t)/2\otimes \Leb^1$ so that $J(\gamma^*)=0$ and $\gamma^*$ is optimal. If $\theta$ was an optimal delay for \pref{optidev} then one should have $F(\theta)=0$ so that $\theta(t)\in \{0, t\}$ a.e. and  $x_\theta=x^*$ (where slightly abusing notations $x_\theta$ denotes $x_{\delta_\theta\otimes \Leb^1}$). From the state equation one should then also have 
\[e^{t/2}=x^*(\theta(t))=-1+2e^{\theta(t)/2}, \mbox{ a.e.}\]
which contradicts  $\theta(t)\in \{0, t\}$ a.e..  

\section{Optimality conditions}\label{pmp}

\subsection{A maximum principle}

In this section, we look for optimality conditions for \pref{optigamma}. In what follows, we further assume that $j$, $h$ and $g$ are continuous and  that 
\begin{itemize}
\item $h$ is of class $C^1$,
\item $j(t,.)$ and $f(t,.)$ are differentiable for every $t\in[0,1]$ and $\nabla_x j$ and $D_x f$ are continuous on $[0,1]\times \R^d$.
\end{itemize}
Let $\gamma=\nu_t\otimes\Leb^1$ be a solution to \pref{optigamma}, $\eta=\mu_t\otimes\Leb^1\in\Gamma$ and $\eps\in(0,1)$; then 
\begin{equation}\label{optizero}
\frac{1}{\eps} [J(\gamma+\eps(\eta-\gamma))-J(\gamma)]\geq 0.
\end{equation}
As usual, our aim is to pass to the limit as $\eps\to 0^+$ and to express the optimality conditions obtained this way in the form of some maximum principle, which will be achieved by introducing some suitable adjoint variable. To shorten notation, we set $\gamma_\eps=\gamma+\eps(\eta-\gamma)$, $x=x_{\gamma}$ and  $x_\eps=x_{\gamma_\eps}$.

\begin{lem}\label{linlem}
As $\eps\to0^+$, $z_\eps:=\eps^{-1}(x_\eps-x)$ converges uniformly on $[0,1]$ to $z$, the solution of the linearized system:
\begin{equation}\label{lin}
\dot{z}(t)=\<Az,\nu_t>+\<a,\mu_t-\nu_t>,\qquad t\in[0,1],\ z(0)=0
\end{equation}
where $A$ and $a$  are the continuous functions:
\begin{equation}\label{Aa}
A(s):=D_x f(s,x(s)),\quad a(s):=f(s,x(s)),\quad\forall s\in[0,1].
\end{equation}
\end{lem}

\begin{proof}
First, let us remark that by the same arguments used in proposition \pref{cauchylip}, the equation \pref{lin} posseses a unique (Lipschitz) solution $z$. By construction we have $z_\eps(0)=0$ and 
\begin{equation}\label{differencez}
\dot{z_\eps}(t)=\<\frac{1}{\eps}[f(., x+\eps z_\eps)-f(., x))], \nu_t>+\<f(., x_\eps), \mu_t-\nu_t>
\end{equation}
from which, by the Lipschitz assumption on $f$, we easily deduce that $z_\eps$ is bounded in $W^{1, \infty}$ (in particular $\Vert x_\eps-x\Vert_{W^{1,\infty}}=O(\eps)$). By Ascoli-Arzela's theorem, $z_\eps$ then possesses  a cluster point $y$ in $C^0$. Passing to the limit in \pref{differencez} along a convergent subsequence  then easily yields $y=z$ and thus by compactness the whole  family $(z_\eps)_\eps$ converges to $z$ as $\eps\to 0^+$. 
\end{proof}

Under our differentiability assumption, using lemma \ref{linlem} and passing to the limit in \pref{optizero} we then get:
\begin{equation}\label{opti1}
\int_0^1 B\cdot z\,dt + b\cdot z(1) + \int_0^1\<g(t,.),\mu_t-\nu_t>\,dt\ge0.
\end{equation}
where $z$ is related to $\gamma$ and $\eta$ by the linearized equation \pref{lin} and
\begin{equation}\label{Bb}
B(s):=\nabla_x j(s,x(s))\ \forall s\in[0,1],\qquad b:=\nabla h(x(1)).
\end{equation} 
To make condition \pref{opti1} tractable, we shall introduce, as usual in control theory, an adjoint state; to do so we shall need a few notations and preliminaries. Recall that $\gamma\in \Gamma$ is given by the disintegration $\gamma=\nu_t \otimes \Leb^1$ and let $\nu:={\pi_2}_\#\gamma$ be the second marginal of $\gamma$ defined by
\[\int_0^1\varphi(s)\,d\nu(s)
=\int_{[0,1]^2}\varphi(s)\,d\gamma(t,s)
=\int_0^1\<\varphi,\nu_t>\,dt,\qquad\forall\varphi\in C^0([0,1]).\]
Invoking again the disintegration theorem, $\gamma$ admits the disintegration $\gamma=\nu\otimes\nu_s^*$ that is $\nu_s^*$ is a Borel family of probability measures and for every test-function $\varphi\in C^0([0,1]^2)$ one has
\begin{equation}\label{nuetoile}
\int_{[0,1]^2}\varphi\,d\gamma=\int_0^1\<\varphi(., s),\nu_s^*>\,d\nu(s)
=\int_0^1\<\varphi(t,.),\nu_t>\,dt.
\end{equation}
Note that the requirement that $\gamma$ is supported by $T$ implies that $\nu_s^*([s,1])=1$. Since $\gamma$ has $\Leb^1$ as first marginal, \pref{nuetoile} also holds for test functions of the form $(t,s)\mapsto q(t) \varphi(s)$ with $\varphi$ continuous but $q$ only $L^1$.  This enables us, for $q\in L^1$, to define $\<q,\nu_s^*>\nu$ as the finite measure defined by
\[\langle\varphi,\<q,\nu_s^*>\nu\rangle:=\int_0^1q(t)\<\varphi,\nu_t>\,dt
=\int_{[0,1]^2}q(t)\varphi(s)\,d\gamma(t,s)\quad\forall\varphi\in C^0([0,1]).\]
The adjoint of \pref{lin} will then be expressed using the measures $\nu_s^*$ and $\nu$, as usual, it is an equation of backward type and reads as:
\begin{equation}\label{adj}
\dot{q}=B-A^T\<q,\nu_s^*>\nu,\qquad q(1)=-b
\end{equation}
where the continuous functions $B$ and $A$ are defined by \pref{Bb} and \pref{Aa} respectively, $A^T$ denotes the transpose of $A$ and the measure  $\<q, \nu_s^*>\nu$ is defined as above. A  solution of \pref{adj} is then by definition an $L^1([0,1],\R^d)$ function $q$ such that
\begin{equation}\label{intformadj}
q(t)=-b-\int_t^1 B\,ds + \int_{[t,1]}A^T(s)\<q,\nu_s^*>\,d\nu(s)
\qquad\mbox{for a.e. }t\in[0,1].\end{equation}
Such a solution is of course ${\rm{BV}}$ and the weak formulation of \pref{adj} reads as
\begin{equation}\label{weakformadj}
\int_0^1 q\cdot\dot{\varphi}\,dt
=-\int_0^1 B\cdot\varphi\,dt-b\cdot\varphi(1)
+\int_{[0,1]^2}\Big(A(s)\varphi(s)\cdot q(t)\Big)\,d\gamma(t,s)
\end{equation}
for every $\varphi \in C^1([0,1]), \R^d)$ such that, $\varphi(0)=0$. Note that it is easy to see in this case that \pref{weakformadj} also holds for every test-function, $\varphi\in W^{1,\infty}([0,1],\R^d)$ such that $\varphi(0)=0$. The well-posedness of \pref{adj} follows from the result below.

\begin{lem}\label{adjlem}
Let $B\in C^0([0,1],\R^d)$, $A\in C^0([0,1],\R^{d\times d})$ and $b\in\R^d$. Then \pref{adj} admits a unique solution $q\in\rm{BV}$. Moreover if $z$ is the solution of \pref{lin} then
\begin{equation}\label{linadj}
\int_0^1 B\cdot z\,dt + b\cdot z(1)=-\int_0^1 q(t)\<a,\mu_t-\nu_t>\,dt.
\end{equation} 
\end{lem}

\begin{proof}
Take $\lambda>\|A\|_{\infty}$ and equip $L^1((0,1), \R^d)$ with the norm
\[\|q\|:=\int_0^1 e^{\lambda t}|q(t)|\,dt.\]
In view of \pref{intformadj}, it is natural, for every $q\in L^1$ to define $Tq$ by
\[Tq(t):=-b-\int_t^1 B\,ds +\int_{[t,1]}A^T(s)\<q,\nu_s^*>\,d\nu(s)\quad\mbox{for a.e. }t\in[0,1],\]
so that $Tq\in{\rm{BV}}$. Let $q_1$ and $q_2$ be in $L^1$ and $t\in (0,1)$, let $\eps\in (0,t)$ and $\varphi_\eps$ be a continuous function such that $0\le\varphi_\eps\le1$, $\varphi_\eps=1$ on $[t,1]$ and $\varphi_\eps=0$ on $[0,t-\eps]$. We then have:
\[\begin{split}
|Tq_1(t)-Tq_2(t)|&\le\|A\|_{\infty}\int_0^1\varphi_\eps\<|q_1-q_2|,\nu_s^*>\,d\nu\\
&=\|A\|_{\infty}\int_0^1\<\varphi_\eps,\nu_\tau>|q_1(\tau)-q_2(\tau)|\,d\tau\\
&\le\|A\|_{\infty}\int_0^1\nu_\tau([t-\eps,\tau])|q_1(\tau)-q_2(\tau)|\,d\tau\\
&\le\|A\|_{\infty}\int_{t-\eps}^1|q_1-q_2|\,d\tau.
\end{split}\]
Letting $\eps\to0^+$ we get
\begin{equation}\label{ineqforT}
|Tq_1(t)-Tq_2(t)|\le\|A\|_{\infty}\int_t^1|q_1-q_2|\,d\tau.
\end{equation}
Multiplying \pref{ineqforT} by $e^{\lambda t}$, integrating and using Fubini's theorem then yields
\[\begin{split}
\|Tq_1-Tq_2\|&\le\|A\|_{\infty}\int_0^1 e^{\lambda t}\Big(\int_t^1|q_1-q_2|\,d\tau\Big)\,dt\\
&=\|A\|_{\infty}\int_0^1|q_1(\tau)-q_2(\tau)|\Big(\int_0^\tau e^{\lambda t}\,dt\Big)\,d\tau\\
&\le\frac{\|A\|_{\infty}}{\lambda}\|q_1-q_2\|.
\end{split}\]
Since we have chosen $\lambda>\|A\|_{\infty}$, the map $T$ is  a contraction and the existence and uniqueness of a solution to \pref{adj} then follows from Banach-Picard's fixed point theorem.

Let us now establish \pref{linadj}. Using $z$ (which is Lipschitz) as test-function in \pref{weakformadj}, we first get
\[\int_0^1 q\cdot\dot{z}\,dt=-\int_0^1 B\cdot z\,dt-b\cdot z(1)+\int_{[0,1]^2} (A(s) z(s))\cdot q(t)\,d\gamma(t,s)\]
then using \pref{lin} yields
\[\int_0^1 q\cdot\dot{z}\,dt=\int_{[0,1]^2}(A(s)z(s))\cdot q(t)\,d\gamma(t,s)+\int_0^1 q(t)\cdot\<a,\mu_t-\nu_t>\,dt\]
which proves \pref{linadj}.
\end{proof}

\begin{thm}\label{pmpontr}
Under the assumptions of this section, if $\gamma\in \Gamma$, $\gamma=\nu_t\otimes \Leb^1=\nu\otimes \nu_s^*$ solves \pref{optigamma} and $x=x_\gamma$ then one has for $\Leb^1$-a.e. $t$
\begin{equation}\label{maxpct}
\spt(\nu_t)\subset{\rm{argmax}}_{s\in [0,t]}\big\{q(t)\cdot f(s,x(s))-g(t,s)\big\},
\end{equation}
where $q$ is the adjoint variable whose dynamics is given by
\begin{equation}\label{adjoint}
q(1)=-\nabla h(x(1)),\ \dot{q}(s)=\nabla_x j(s,x(s))-D_xf(s,x(s))^T \<q,\nu_s^*>\nu(s).
\end{equation}
\end{thm}

\begin{proof}
Let $\gamma=\nu_t\otimes\Leb^1\in\Gamma$ be optimal for \pref{optigamma}; then for every $\eta=\mu_t\otimes\Leb^1$, it follows from \pref{opti1}, \pref{Bb} and \pref{linadj} that
\[\int_0^1\<q(t)\cdot f(.,x(.))-g(t,.),\nu_t-\mu_t>\,dt\ge0\]
with $q$ defined by \pref{adjoint}. Since in the previous inequality, $\mu_t$ is an arbitrary probability measure supported on $[0,t]$, \pref{maxpct} directly follows.
\end{proof}

\begin{rem}\label{csuff} 
In the case where $f$ is linear in $x$ (that is $f(s,x)=A(s) x$) and $j(t,.)$ and $h$ are convex (so that $J(\gamma)$ is convex in $x_\gamma$, but in general not in $\gamma$), then the following condition is sufficient  for optimality: 
\begin{equation}\label{maxpcsuff}
\spt(\nu_t)\subset{\rm{argmax}}_{s\in [0,t]}\big\{q(t)\cdot (A(s) x_\eta(s))-g(t,s)\big\},
\end{equation}
for every $\eta\in \Gamma$ and a.e. $t\in [0,1]$. Indeed, if $\gamma=\nu_t \otimes \Leb^1$, $x=x_\gamma$,  $q$ is the adjoint variable defined as in theorem \ref{pmpontr} and $\eta=\mu_t\otimes \Leb^1 \in \Gamma$, by the above convexity assumptions and the same manipulations as before (integration by parts and disintegration), we get:
\[\begin{split}
J(\eta)-J(\gamma)&\geq \int_0^1 \nabla_x j(s,x(s))\cdot (x_\eta(s)-x(s))ds + \int_{[0,1]^2} g(t,s) d(\eta-\gamma)(t,s)\\
&= \int_0^1 \<q(t)\cdot Ax_\eta(.)-g(t,.), \nu_t-\mu_t> dt
\end{split}\] 
which proves that \pref{maxpcsuff} is a sufficient optimality condition. Of course, \pref{maxpcsuff} is not necessary in general and is difficult to exploit since it involves every admissible $\eta$. However, we shall give a simple example in section \ref{exampl} where the sufficient condition \pref{maxpcsuff} together with the necessary condition of theorem \ref{pmpontr} actually enables one to compute an optimal control. 
\end{rem}

\begin{rem}\label{moregeneral}
Let us indicate, without giving details, that one may obtain in a similar way as in theorem \ref{pmpontr} optimality conditions for the optimal control of the slightly more general state equation containing a local term:
\[\dot{x}(t)=\eta(t, x(t))+\<f(., x(.)), \nu_t>, \; x(0)=x_0.\]
Indeed, considering the minimization of the cost $J$ (which has the same form as before), one gets, by similar arguments as previously, that any minimizer $\gamma=\nu_t \otimes \Leb^1$ has to satisfy the statement of theorem \ref{pmpontr}, the only modification being that the adjoint equation now reads as
\begin{equation}\label{adjointmoregeneral}
 \dot{q}(s)=\nabla_x j(s,x(s))-D_x \eta(s,x(s))^T q(s)-D_xf(s,x(s))^T \<q,\nu_s^*>\nu(s).
\end{equation}

\end{rem}

\subsection{Solutions with finitely many Dirac masses}

This paragraph is somehow independent from the previous one on optimality conditons. A simple application of Carath\'eodory's theorem, gives the existence of optimal controls involving  only finitely many Dirac masses.

\begin{prop}
Assume that $f$ satisfies \pref{flip} and $f$ and $g$ are continuous, then  for every $\gamma=\nu_t\otimes\Leb^1\in\Gamma$ there exists $\eta=\mu_t\otimes \Leb^1$ such that $\mu_t$ has finite support with at most cardinality $d+2$ such that $x_{\gamma}=x_{\eta}$ and
\[\int_0^1 g(t,s)\,d\nu_t(s)=\int_0^1 g(t,s)\,d\mu_t(s)\quad\mbox{ a.e. }t.\]
In particular, in the optimization problem \pref{optigamma} it is enough to optimize over measures in $\Gamma$ of the form $\nu_t\otimes \Leb^1$ with $\nu_t$ a convex combination of at most $d+2$ ($d+1$ if $g=0$) Dirac masses. 
\end{prop}

\begin{proof}
Let $\gamma\in\Gamma$, $x:=x_{\gamma}$ and define for every $t\in[0,1]$ the compact set
\[S_{x,t}:=\big\{(f(s,x(s)),g(t,s))\ :\ s\in [0,t]\big\}\subset\R^{d+1}.\]
By construction, for almost every $t$ one has
\[\Big(\dot{x}(t),\int_0^t g(t,s)\,d\nu_t(s)\Big)\in{\rm{co}}(S_{x,t}).\]
It thus follows from Carath\'eodory's theorem (see \cite{ro}) that $\big(\dot{x}(t),\int_0^t g(t,s)\,d\nu_t(s)\big)$ may be expressed as a convex combination of at most $d+2$ points in $S_{x,t}$. Hence there exists a discrete probability $\mu_t$ on $[0, t]$ with at most $d+2$ points in its support such that
\[\dot{x}(t)=\<f(., x(.)),\mu_t>,\qquad
\int_0^1 g(t,s)\,d\nu_t(s)=\int_0^1 g(t,s)\,d\mu_t(s).\]
The fact that the discrete measure $\mu_t$ can be chosen measurable in $t$ follows from standard measurable selection arguments (see for instance \cite{cv}).
\end{proof}

\section{Examples}\label{exampl}

\subsection{Scalar case}

We start by considering the scalar ODE with memory
$$\dot x=\alpha\<x,\nu_t>,\qquad x(0)=1$$
where $\alpha>0$ is a parameter and $\nu_t$ is an admissible control. It is clear that if we want to minimize a cost like
$$J(\gamma)=\int_0^1x_\gamma(t)\,dt$$
the best choice for the control is $\gamma=\nu_t \otimes \Leb^1$ with $\nu_t=\delta_0$ for all $t\in[0,1]$, which gives
$$x(t)=1+\alpha t,\qquad J_{min}=1+\frac{\alpha}{2}.$$
Analogously, for a cost like
$$J(\gamma)=-x_\gamma(1)$$
the best choice is  $\gamma=\nu_t \otimes \Leb^1$ with $\nu_t=\delta_t$ for all $t\in[0,1]$, which gives
$$x(t)=e^{\alpha t},\qquad J_{min}=-e^\alpha.$$
Take now the cost
$$J(\gamma)=a\int_0^1x_\gamma(t)\,dt-bx_\gamma(1)$$
with $a,b>0$ and with no penalization in the use of memory (i.e. $g=0$). The Pontryagin principle of theorem \ref{pmpontr} gives the adjoint equation
\begin{equation}\label{adjex1d}
\dot q(t)=a-\alpha\<q,\nu^*_t>\nu(t),\qquad q(1)=b
\end{equation}
and the necessary condition of optimality
\begin{equation}\label{max1d}
\spt(\nu_t)\subset\argmax_{s\in[0,t]}\big\{\alpha q(t)x_\gamma(s)\big\}
\end{equation}
since $x_\gamma$ is increasing, the latter is equivalent to
$$\nu_t=\left\{\begin{array}{ll}
\delta_t&\hbox{if }q(t)>0\\
\delta_0&\hbox{if }q(t)<0.\\
\end{array}\right.$$
Thanks to remark \ref{csuff} and the fact that $x_\eta$ is nondecreasing for every admissible $\eta$ then the conditions \pref{adjex1d}-\pref{max1d} are in fact \emph{sufficient}. Since $q(1)=b>0$, we necessarily have $\nu_t=\delta_t$ for $t$ close to $1$. Now, it is natural, in view of the previous considerations to look for an optimal control in the form
$$\nu_t=\left\{\begin{array}{ll}
\delta_0 &\hbox{if } t\in (0, t_0) \\
\delta_t &\hbox{if } t\in (t_0,1)\\
\end{array}\right.$$
where $t_0\in [0,1]$ will be determined in such a way that the previous necessary and sufficient optimality conditions hold.  For such a family of Dirac masses, a direct computation yields
\[\nu=t_0 \delta_0 + \Leb^1_{[t_0, 1]}\]
and 
 $$\nu_s^*=\left\{\begin{array}{lll}
 t_0^{-1} \Leb^1_{[0, t_0]} &\hbox{if } s=0 \hbox{ and } t_0\neq 0,\\
 \delta_0 &\hbox{if } s=0 \hbox{ and } t_0= 0,\\
\delta_t &\hbox{if } s\in (t_0,1)\\
\end{array}\right.$$
(note that $\nu_s^*$ is only defined for $\nu$-a.e. $s$). Integrating \pref{adjex1d} on $(t_0, 1]$ then yields
\[q(t)=\frac{a}{\alpha} +\Big(b-\frac{a}{\alpha}\Big) e^{\alpha(1-t)}, \; t\in (t_0,1).\]
There are now two cases: either 
\begin{equation}\label{caslimite}
\frac{b}{a} \geq \frac{e^{\alpha}-1}{\alpha e^{\alpha}}
\end{equation} 
in which case, $q$ remains positive on $(0,1)$ and thus $t_0=0$ and $\nu_t=\delta_t$ for every $t\in [0,1]$ is optimal. Or \pref{caslimite} does not hold; in this case,  set
\[t_0=1-\frac{1}{\alpha} \ln \Big(\frac{a}{a-\alpha b}\Big)\in (0,1).\]
Now let $t\in (0,t_0)$, integrating \pref{adjex1d} between $t$ and $t_0$, we simply get $q(t)=a(t-t_0)$ so that $q$ is negative on $(0,t_0)$ and then $\nu_t=\delta_0$ for $t\in (0, t_0)$ and $\nu_t=\delta_t$ for $t\in (t_0,1)$ is an optimal control.

\subsection{A two-dimensional example}

Let us now consider the two-dimensional linear-quadratic-like problem that consists in minimizing
\[\int_0^1 \frac{1}{2}\left( a(t)x^2(t) +b(t)y^2(t) \right) dt\]
with respect to the admissible control $(\nu_t)$ when the state equations for $x$ and $y$ are  
\[\dot{x}(t)=\alpha \<x, \nu_t>, \; \dot{y}(t)=\beta \<y, \nu_t>\]
together with the initial conditions $(x(0), y(0))=(x_0, y_0)$. The weight functions $a$ and $b$ are continuous  (but not necessarily positive).  We also assume that $\alpha>0$, $\beta>0$ and $x_0>0$, $y_0<0$ which guarantees that $x$ is increasing and positive and $y$ is decreasing and negative.  The adjoint equations read as
\begin{equation}\label{adjex2d}
\dot{q}_1=ax -\alpha \<q_1, \nu_s^*>\nu, \; \dot{q}_2=by -\alpha \<q_2, \nu_s^*>\nu, \; q_1(1)=q_2(1)=0.
\end{equation}
Furthermore, the maximum principle of theorem \ref{pmpontr} gives
\[\spt(\nu_t)\subset{\rm{argmax}}_{s\in [0,t]}\big\{q_1(t) x(s) +q_2(t) y(s)\big\}.\]
Since $x$ is increasing and $y$ is decreasing, we thus deduce that:
\[\begin{split}
q_1(t)>0, \; q_2(t)<0 \Rightarrow \nu_t=\delta_t,\\
q_1(t)<0, \; q_2(t)>0 \Rightarrow \nu_t=\delta_0.
\end{split}\]
Now, if $a(1)>0$ and $b(1)>0$, it is easy to deduce from \pref{adjex2d} and the fact that $\spt(\nu_s^*)\subset [s,1]$ (and that $\nu(\{1\})=0$) that $q_1(t)<0$ and $q_2(t)>0$ for $t<1$ sufficiently close to $1$ and therefore the optimal $\nu$'s have to satisfy $\nu_t=\delta_0$ for $t$ close to $1$.  

\smallskip

If $a$ and $b$ are everywhere positive, then \pref{adjex2d} implies $q_1<0$ and $q_2>0$ on $[0,1)$ so that (not surprisingly) there is only one optimal $\nu$ that is $\nu_t=\delta_0$ for every $t$. Now let us consider the case where $a$ and $b$ may change sign and let us look for conditions that ensure that $q_1(t)>0$ and $q_2(t)<0$ for $t$ close to $0$ so that optimal $\nu$'s satisfy $\nu_t=\delta_t$ for $t$ close to $0$. 

\smallskip

Let us assume further that $\alpha<1$ and $\beta<1$, then we straightly deduce from the state equation the estimates
\begin{equation}\label{estimx}
\Vert x \Vert_{\infty}\leq \frac{x_0}{1-\alpha}, \; \Vert y \Vert_{\infty}\leq \frac{\vert y_0\vert}{1-\beta}
\end{equation} 
using those estimates and the adjoint equation for $q_1$, we also get
\begin{equation}\label{estimq1}
\Vert q_1\Vert_{\infty}\leq \frac{\Vert a \Vert_{\infty} x_0}{(1-\alpha)^2}. 
\end{equation}
We then get
\[\begin{split}
-q_1(0)&=\int_0^1 ax\,dt-\alpha\int_0^1 \<q_1, \nu_s^*> d\nu\\
&\le\int_0^1 a_+x\,dt -\int_0^1 a_-x\,dt +\frac{\alpha\|a\|_{\infty} x_0}{(1-\alpha)^2}\\
&\le x_0\Big(\frac{1}{1-\alpha}\int_0^1 a_+\,dt -\int_0^1 a_-\,dt +\frac{\alpha\|a\|_{\infty}}{(1-\alpha)^2}\Big)
\end{split}\]
so that $q_1(t)<0$ for $t$ close to $0$ as soon as
\[\int_0^1 a_-\,dt > \frac{1}{1-\alpha}\int_0^1 a_+\,dt +\frac{\alpha\|a\|_{\infty}}{(1-\alpha)^2}.\]
A similar condition on $b$ ensures that $q_2(t)>0$ for $t$ close to $0$.


\begin{thebibliography}{999}

\bibitem{amb}L.~AMBROSIO: {\it Lecture notes on optimal transport problems.} In ``Mathematical aspects of evolving interfaces'', Lecture Notes in Mathematics {\bf1812}, Springer-Verlag, Berlin (2003), 1-52.

\bibitem{bc}  R. BELLMAN, K.L. COOKE: {\it Differential-difference equations},  Mathematics in Science and Engineering, Academic Press, New York-London (1963).

\bibitem{blibro}G.~BUTTAZZO: {\it Semicontinuity, Relaxation and Integral Representation in the Calculus of Variations.} Pitman Res. Notes Math. Ser. {\bf207}, Longman, Harlow (1989).

\bibitem{cht}G.~CARLIER, A.~HOUMIA, R.~TAHRAOUI: {\it On Pontryagin's principle for the optimal control of some state equation with memory,} To appear in J. Convex Analysis, {\bf17} no. 3 (2010).


\bibitem{ct}G.~CARLIER, R.~TAHRAOUI: {\it On some optimal control problems governed by a state equation with memory,} ESAIM Control Optim. Calc. Var., {\bf14} no. 4 (2008), 725--743.


\bibitem{cthjb}G.~CARLIER, R.~TAHRAOUI: {\it Hamilton-Jacobi-Bellman equations for the optimal control of a state equation with memory,} to appear in ESAIM Control Optim. Calc. Var..

\bibitem{cv} C.~CASTAING, M.~VALADIER: {\it Convex Analysis and Measurable Multifunctions.} Lecture Notes in Mathematics {\bf580}, Springer-Verlag, Berlin (1977). 

\bibitem{dm} C. DELLACHERIE, P.-A. MEYER: {\it Probabilities and Potential}, Mathematical Studies {\bf 29}, North-Holland (1978).

\bibitem{jkt} E. JOUINI, P.F. KOEHL, N. TOUZI: {\it Optimal investment with taxes: an optimal control problem with endogeneous delay,} Nonlin. Anal. TMA {\bf37} (1999) 31--56.

\bibitem{ro} R.T.~ROCKAFELLAR: {\it Convex Analysis.} Princeton University Press, Princeton (1972).

\bibitem{st}  L.~SAMASSI, R.~TAHRAOUI: {\it How to state necessary optimality conditions for control problems with deviating arguments?},  ESAIM Control Optim. Calc. Var.  {\bf{14}}  no. 2  (2008), 381--409.


\end{thebibliography}
\end{document}